\documentclass{amsart}
\usepackage{graphicx}
\usepackage{amsmath}
\usepackage{amssymb}
\usepackage{amsthm}
\usepackage{esint}
\usepackage{mathrsfs}

\newtheorem{theorem}{Theorem}[section]

\newtheorem{lemma}[theorem]{Lemma}

\numberwithin{equation}{section}

\newcommand{\abs}[1]{\lvert#1\rvert}

\title{
New curvature characterization for 
the weighted Sasaki
sphere}

\author{Pak Tung Ho}
\address{Department of Mathematics, Tamkang University, Tamsui, New Taipei City 251301, Taiwan}

\email{paktungho@yahoo.com.hk}

\author{Kuang-Ru Wu}
\address{Erdős Center, HUN-REN Rényi Institute, Reáltanoda utca 14, H-1053, Budapest, Hungary}

\email{wuuuruuu@gmail.com}

\newcommand{\RN}[1]{%
  \textup{\uppercase\expandafter{\romannumeral#1}}%
}

\usepackage{fullpage}
\usepackage{multicol}

\begin{document}

\date{21st of August, 2025}

\parskip=6pt

\begin{abstract}
In this paper, we use the shifted cones
introduced by Yang and Zhang  to classify Sasaki manifolds.
This gives a  new curvature characterization for 
the weighted Sasaki
sphere.

\end{abstract}

\maketitle

\section{Introduction}

The study of the relationship between curvature and topology
of a manifold is one of the most important subjects in Riemannian geometry.
One of the examples 
is the Bonnet–Myers theorem, which says that 
a complete  Riemannian manifold 
such that its Ricci curvature bounded 
below by a positive constant 
must be compact. 
Another example 
is the so-called sphere theorem, 
which roughly says that 
a manifold must be the standard sphere 
under certain curvature condition. 
Here we mention some results of this type:
In \cite{MarioMoore},  
Micallef and Moore showed
that a compact simply connected $n$-dimensional Riemannian manifold which has positive isotropic curvature, where $n\geq 4$, 
must be homeomorphic to a sphere. 
Using Ricci flow, Hamilton proved in \cite{Hamiton}
that
a compact $3$-dimensional manifold which admits a Riemannian metric with strictly positive Ricci curvature must be diffeomorphic to 
the $3$-dimensional standard sphere. 
Still using Ricci flow, 
Brendle and Schoen proved in \cite{BrendleSchoen}
the differentiable sphere theorem, 
which says that 
a compact Riemannian manifold of dimension $n\geq 4$ with $1/4$-pinched sectional curvature
admits a metric of constant  curvature and therefore is diffeomorphic to a spherical space form. 

For K\"ahler manifolds, we have the Frankel conjecture, which asserts that 
 a compact K\"ahler manifold with positive bisectional curvature  must be biholomorphic to complex projective space.
The Frankel conjecture was proved by 
Siu and Yau in \cite{SiuYau}
and independently by Mori
in \cite{Mori}.
Later, the following generalized Frankel conjecture 
was proved in \cite{Mok}
by Mok:
Suppose that $M$ is a compact K\"ahler manifold of nonnegative  bisectional curvature. Then the universal cover of $M$, with its natural metric, is biholomorphically isometric to the metric product of complex Euclidean space, with some number of irreducible closed hermitian symmetric spaces with rank larger than one, with the product of some number of complex projective spaces, each of which has a K\"ahler metric of nonnegative  bisectional curvature.

Very recently, Yang and Zhang 
introduced in \cite{YangZhang} a new positivity notion for curvature of
Riemannian manifolds and obtain characterizations for spherical space forms and
the complex projective space $\mathbb{CP}^n$.
To state their results, 
we introduce the following notations: 
Let $\Lambda=(a_1,...,a_N)\in\mathbb{R}^N$
and $\sigma_k$ be the $k$-th 
elementary symmetric function
$$\sigma_k(\Lambda)
=\sum_{1\leq i_1<\cdots<i_k\leq N}
a_{i_1}\cdots a_{i_k}.$$
The $\Gamma_k^+$ cone
is defined as
$$\Gamma_k^+:=
\{\Lambda\in\mathbb{R}^N: 
\sigma_j(\Lambda)>0\mbox{ for all }j\leq k\}.$$
Let $\overline{\Gamma}^+_k$
be the closure of $\Gamma_k^+$. 
Hence, $\Lambda\in \overline{\Gamma}^+_k$
if and only if $\sigma_j(\Lambda)\geq 0$ 
for all $j\leq k$. 
It is clear that 
$\Gamma_N^+\subset\cdots\subset \Gamma_1^+$
and $\Lambda=(a_1,...,a_N)\in \Gamma_N^+$
if and only if $a_i>0$ for all $i$. 
In \cite{YangZhang}, 
Yang and Zhang studied 
a version of shifted cone on 
Riemannian manifolds and K\"ahler manifolds. 
In particular, we mention the following 
theorems related 
to our main results. 

Let $(M,g)$ be a K\"ahler manifold of complex dimension $n$. It is well-known that
by the $J$-invariant property, the curvature operator of the underlying Riemannian
manifold $(M,g)$ vanishes on the orthogonal complement of the holonomy algebra
$\mathfrak{u}(n)\subset \mathfrak{so}(2n)$. Hence, the Riemannian curvature operator of a K\"ahler manifold has
a kernel of real dimension at least $n(n-1)$. We consider the K\"ahler curvature
operator $R:=R|_{\mathfrak{u}(n)}$. Let
$$\rho_1\leq\cdots\leq\rho_{n^2}$$
be the real eigenvalues
of the K\"ahler 
curvature operator. 
Let 
$\Lambda=(\rho_1,\cdots,\rho_{n^2})$
and $T=\rho_1+\cdots+\rho_{n^2}$. 
For $k\in \{1,2,...,n^2-1\}$, 
we define 
\begin{equation}
\beta_k:=\frac{1}{n^2}
-\frac{1}{n^2}
\sqrt{\frac{k}{(n^2-1)(n^2-k)}}    
\end{equation}
and the shifted point 
\begin{equation}
\Lambda_{\beta_k}:=\Lambda-\beta_k
(T,\cdots,T)\in\mathbb{R}^{n^2}.
\end{equation}
We say $g\in \Gamma_2^+(\beta_k)$
if $\Lambda_{\beta_k}\in\Gamma_2^+$
for each point $x\in M$. 
The following two theorems were proved
by Yang and Zhang. 

\begin{theorem}[Theorem 1.5 in \cite{YangZhang}]
\label{YZthm5}
Let $(M, g)$ be a compact K\"ahler manifold with complex dimension $n\geq 2$.
If $g\in \Gamma_2^+(\beta_2)$, 
then $M$  is biholomorphic to $\mathbb{CP}^n$. \end{theorem}

\begin{theorem}[Theorem 1.6 in \cite{YangZhang}]
\label{YZthm6}
Let $(M,g)$
be a compact K\"ahler manifold with complex dimension $n\geq 2$.
If $g\in \overline{\Gamma}_2^+(\beta_2)$,
then one of the following holds:
\begin{enumerate}
    \item $M$ is biholomorphic
    to $\mathbb{CP}^1\times\mathbb{CP}^1$;
    
    \item  $(M,g)$ is a flat torus;

      \item $M$ is biholomorphic to $\mathbb{CP}^n$.
\end{enumerate}    
\end{theorem}

Sasaki geometry is an odd dimensional companion of K\"ahler geometry. 
The counterpart of the Frankel conjecture
for Sasaki manifolds has been proved by 
He and Sun in \cite{HeSunAdv}
(see Theorem \ref{thm He--sun positive}
below), 
while the counterpart of the generalized Frankel conjecture  for Sasaki manifolds
has been proved by He and Sun in \cite{HeSunIMRN}
(see Theorem \ref{thm He--Sun nonnegative}
below).

The aim of the paper is to use the shifted cones introduced in \cite{YangZhang} to classify Sasaki manifolds. Let $(M,g)$ be a Sasaki manifold of dimension $2m+1$. The metric $g$ induces a transverse K\"ahler metric which is a collection of K\"ahler metrics $g^T_\alpha$ on $V_\alpha\subset \mathbb{C}^m$ with $\alpha\in A$ (we will recall the definition of transverse K\"ahler metrics in Section \ref{sec prelim} with more details). We denote the K\"ahler curvature operator of  $g^T_\alpha$ by $R^T_\alpha$. Let $\rho^\alpha_1\leq\cdots\leq \rho^\alpha_{m^2}$ be the eigenvalues of the K\"ahler curvature operator $R^T_\alpha$ on $V_\alpha$. For $1\leq k\leq m^2-1$, we define 
$$\beta_k:=\frac{1}{m^2}-\frac{1}{m^2}\sqrt{\frac{k}{(m^2-1)(m^2-k)}}$$
and we say $g^T_\alpha \in \Gamma^+_2(\beta_k)$ at $x\in V_\alpha$ if $\Lambda_{\beta_k}\in \Gamma^+_2$ at $x$
 where $\Lambda_{\beta_k}=(\rho_1^\alpha,\cdots,\rho_{m^2}^\alpha)-\beta_k(T,\cdots, T)$ and $T=\rho^\alpha_1+\cdots \rho^\alpha_{m^2}$. Namely, the following is true at $x$:
\begin{align}
  &\sum_{1 \leq j \leq m^2} (\rho^\alpha_j-\beta_kT)>0 \text{ and } \label{1}\\
&\sum_{1\leq i< j\leq m^2} (\rho^\alpha_i-\beta_kT)(\rho^\alpha_j-\beta_kT)>0\label{2}.  
\end{align}
 If $g^T_\alpha \in \Gamma^+_2(\beta_k)$ for all $x$ in $V_\alpha$ then we say $g^T_\alpha \in \Gamma^+_2(\beta_k)$ in $V_\alpha$.
 
Inspired by 
Theorem \ref{YZthm5} and Theorem \ref{YZthm6}, we prove the following two theorems in this paper.

\begin{theorem}\label{thm 1}
    Let $(M,g)$ be a compact Sasaki manifold of dimension $2m+1$. If $g^T_\alpha \in \Gamma^+_2(\beta_2)$ in $V_\alpha$ for any $\alpha \in A$, then $\pi_1(M)$ is finite and  the universal cover of $(M,g)$ is isomorphic to a weighted Sasaki sphere with a simple metric.
\end{theorem}

\begin{theorem}\label{thm 2}
    Let $(M,g)$ be a compact Sasaki manifold of dimension $2m+1$. If $g^T_\alpha \in \overline{\Gamma}^+_2(\beta_2)$ in $V_\alpha$ for any $\alpha \in A$, then one of the following holds.
\begin{enumerate}
    \item The universal cover of $(M,g)$ is isomorphic to a weighted Sasaki sphere with a simple metric.
    
    \item  The universal cover of $(M,g)$ is isomorphic to  $\mathbb{R}^{2m+1}$ with the standard Sasaki structure.

      \item $M$ is quasi-regular, and the universal cover of the quotient orbifold $M/F_{\xi}$, where $F_{\xi}$ is the Reeb foliation, is isomorphic to $W\mathbb{P}^1\times W\mathbb{P}^1 $ or  $W\mathbb{P}^m$.
\end{enumerate}    
\end{theorem}

Here is the plan of this paper. 
In section \ref{sec prelim}, we collect 
some definitions and facts about Sasaki manifolds. In particular, the concept of a weighted Sasaki sphere with a simple metric will be recalled in Subsection \ref{subsec simple metric}.
In section \ref{sec_proof}, 
we prove Theorem \ref{thm 1} and 
Theorem \ref{thm 2} by following the 
arguments in \cite{YangZhang}.

\section{Preliminary and facts about Sasaki manifolds}\label{sec prelim}

We recall briefly here the definition of a Sasaki manifold and its transverse K\"ahler structure. We will follow closely \cite[Sections 2 and 3]{FutakiOnoWang} (see also \cite{MartelliSparksYau}).   A Sasaki manifold is a Riemannian manifold $(M,g)$ whose cone manifold $(\mathbb{R}_+\times M,dr^2+r^2g)$ is K\"ahler. Hence the dimension of $M$ is an odd number $2m+1$. We denote the Reeb vector field on $M$ by $\xi$ and the contact form by $\eta$.

The transverse K\"ahler structure on the Sasaki manifold $(M,g, \xi, \eta)$ is defined as follows. Let $\{U_\alpha\}_{\alpha\in A}$ be an open cover of $M$ and $\pi_\alpha:U_\alpha\to V_\alpha\subset \mathbb{C}^m$ be submersions such that when $U_\alpha \cap U_\beta\neq \emptyset$, $$\pi_\alpha\circ\pi_\beta^{-1}:\pi_\beta(U_\alpha \cap U_\beta)\to \pi_\alpha(U_\alpha \cap U_\beta) $$ is biholomorphic. Let $D=\ker \eta\subset TM$. There is a canonical isomorphism $$d\pi_\alpha : D_p\to T_{\pi_\alpha(p)}V_\alpha, $$
for any $p\in U_\alpha$. The restriction of the Sasaki metric $g$ to $D$ gives a Hermitian metric $g_\alpha^T$ on $V_\alpha$. The Hermitian metric $g_\alpha^T$ is actually K\"ahler. On the intersection $U_\alpha\cap U_\beta$, the metrics $\{g^T_\alpha\}$ are isometric through the maps $\pi_\alpha\circ\pi_\beta^{-1}$. The collection of K\"ahler metrics $\{g^T_\alpha\}$ on $\{V_\alpha\}$ is called a transverse K\"ahler metric.

\subsection{Weighted Sasaki spheres}\label{subsec simple metric}

Let $(\xi, \eta, g)$ be the standard Sasaki structure on the sphere $S^{2m+1}=\{(z_0,\dots, z_m)\in \mathbb{C}^{m+1}:  \sum_{i=0}^m |z_i|^2=1  \}$. The Reeb vector field and the contact form are, respectively, $$\xi=\sum_{i=0}^m (y_i\frac{\partial }{\partial x_i}-x_i\frac{\partial }{\partial y_i}) \text{ and }   \eta=\sum_{i=0}^my_idx_i-x_idy_i,$$ where $z_i=x_i+\sqrt{-1}y_i$. 
The Riemannian metric $g$ is the Euclidean metric of $\mathbb{C}^{m+1}$ restricted to $S^{2m+1}$. A weighted Sasaki sphere with a simple metric is obtained by deforming the standard sphere $(S^{2m+1},g)$ through a type-$I$ deformation and then a transverse K\"ahler deformation which we define below.

For $a=(a_0,\dots, a_m)\in \mathbb{R}^{m+1}_+$, we define $(\xi_a, \eta_a, g_a)$ with weight $a$ on $S^{2m+1}$ by
\begin{align*}
&\xi_a=\sum_{i=0}^m a_i (y_i\frac{\partial }{\partial x^i}-x_i\frac{\partial }{\partial y^i}), \,\,\, \eta_a=\frac{\eta}{\sum_{i=0}^m a_i|z_i|^2},  \\
&g_a(X,Y)=\frac{g\big(X-\eta_a(X)\xi_a,Y-\eta_a(Y)\xi_a\big)}{\sum_{i=0}^m a_i|z_i|^2}+\eta_a(X) \eta_a(Y),
\end{align*}
where $X,Y$ are vector fields on $S^{2m+1}$. This deformation is called a  type-$I$ deformation and is introduced in \cite{Takahashi} (see also \cite[Example 7.1.12]{BoyerGalickibook}).

To define the transverse K\"ahler deformation, we start with a smooth function $\phi$ on $S^{2m+1}$ with zero Lie derivative $\mathcal{L}_{\xi_a} \phi=0$ (such a function is called basic). A transverse K\"ahler deformation of $(S^{2m+1},\xi_a, \eta_a)$ is given by  $(S^{2m+1},\xi_a, \Tilde{\eta} )$
where $\Tilde{\eta}:=\eta_a+2d^c_B\phi$
(for more details about the operator $d^c_B=i(\bar{\partial}_B-\partial_B)/2$, see \cite[Section 4]{FutakiOnoWang}).

\subsection{Positively curved Sasaki manifolds}

In this subsection, we collect important classification theorems in Sasaki geometry due to He and Sun \cite{HeSunAdv}, and refinements due to Huang \cite{HuangHong}.

\begin{theorem}\label{thm He--sun positive}
    Let $(M, g)$ be a compact Sasaki manifold of dimension $2m+1$ with positive transverse orthogonal bisectional curvature. Then $\pi_1(M)$ is finite and  the universal cover of $(M,g)$ is isomorphic to a weighted Sasaki sphere with a simple metric.
\end{theorem}
Theorem \ref{thm He--sun positive} is proved by He and Sun for positive transverse bisectional curvature in \cite{HeSunAdv}, and extended to the positive transverse \textit{orthogonal} bisectional curvature by Huang in \cite{HuangHong}. For the nonnegative case, He and Sun in \cite[Theorems 1.2]{HeSunIMRN} prove:
\begin{theorem}\label{thm He--Sun nonnegative}
    If $(M,g)$ is a compact Sasaki manifold with nonnegative transverse bisectional curvature, then one of following is true.
  \begin{enumerate}
      \item $M$ is irregular, $\pi_1(M)$ is finite, and the universal cover of $(M,g)$ is isomorphic to a weighted Sasaki sphere with a simple metric of nonnegative transverse bisectional curvature.
    
      \item $M$ is quasi-regular, and the universal cover of the quotient orbifold $M/F_{\xi}$, where $F_{\xi}$ is the Reeb foliation, is isomorphic to $$(W\mathbb{P}^{n_1}, \omega_1)\times \cdots \times (W\mathbb{P}^{n_j},\omega_j)\times (O_1,h_1)\times \cdots \times (O_k,h_k)\times (\mathbb{C}^l,h_0)$$
      where $\omega_1 \sim \omega_j$ are K\"ahler metrics on the weighted projective spaces with nonnegative bisectional curvature, $h_1\sim h_k$ are the canonical metrics on the compact irreducible Hermitian symmetric spaces of rank $\geq 2$, and $h_0$ is the flat metric on $\mathbb{C}^l$.
  \end{enumerate}  
\end{theorem}


\section{Proof of Theorems \ref{thm 1} and \ref{thm 2}}\label{sec_proof}

Let $(M,g)$ be a Sasaki manifold of dimension $2m+1$. Recall that the metric $g$ induces a transverse K\"ahler metric $\{g^T_\alpha\}_{\alpha\in A}$. The K\"ahler curvature operator of  $g^T_\alpha$ is denoted by $R^T_\alpha$. The following is from \cite[Theorem 2.7]{YangZhang}.
\begin{lemma}\label{lemma}
    Let $\rho^\alpha_1\leq\cdots\leq \rho^\alpha_{m^2}$ be the eigenvalues of the K\"ahler curvature operator $R^T_\alpha$ on $V_\alpha$. For $1\leq k\leq m^2-1$,

\begin{enumerate}
    \item If $g^T_\alpha \in \Gamma^+_2(\beta_k)$ at $x\in V_\alpha$, then $\rho^\alpha_1+\cdots +\rho^\alpha_k>0$ at $x$.
    \item If $g^T_\alpha \in \overline{\Gamma}^+_2(\beta_k)$ at $x\in V_\alpha$, then one of the following holds:
\begin{enumerate}
    \item $\rho^\alpha_1+\cdots +\rho^\alpha_k>0$ at $x$.
    \item $\rho^\alpha_1=\cdots =\rho^\alpha_k=0$ and $\rho^\alpha_{k+1}=\cdots =\rho^\alpha_{m^2}\geq 0$ at $x$.
\end{enumerate}
\end{enumerate}
\end{lemma}

Using Lemma \ref{lemma}, we give the poofs of Theorems \ref{thm 1} and \ref{thm 2}.
\begin{proof}[Proof of Theorem \ref{thm 1}]
    If $g^T_\alpha \in \Gamma^+_2(\beta_2)$ in $V_\alpha$ for any $\alpha \in A$, then the curvature  $\{R^T_\alpha \}_{\alpha\in A}$ is 2-positive by Lemma \ref{lemma}. Hence $R^T_\alpha$ has positive orthogonal bisectional curvature by \cite[Lemma 2.10]{YangZhang}. Therefore, the transverse orthogonal bisectional curvature is positive, and according to Theorem \ref{thm He--sun positive}, the proof is complete.
\end{proof}

\begin{proof}[Proof of Theorem \ref{thm 2}]

If  $g^T_\alpha \in \overline{\Gamma}^+_2(\beta_2)$ for any $\alpha \in A$, then we have the following three cases  by Lemma \ref{lemma}: 
\begin{enumerate}
    \item\label{casea} $\rho^\alpha_1+\rho^\alpha_2 \geq 0$ on $V_\alpha$ for any $\alpha \in A$, and there exists a point $x\in V_{\alpha_0}$ for some $\alpha_0 \in A$ such that $\rho^{\alpha_0}_1(x)+\rho^{\alpha_0}_2(x) > 0$.
    \item\label{caseb} $\rho^\alpha_1=\rho^\alpha_2 =\cdots =\rho^\alpha_{m^2}= 0$ on $V_\alpha$ for any $\alpha \in A$. 
    \item\label{casec} $\rho^\alpha_1=\rho_2^\alpha=0$ and $\rho^\alpha_{3}=\cdots =\rho^\alpha_{m^2}\geq 0$  on $V_\alpha$ for any $\alpha \in A$, and there exists a point $x\in V_{\alpha_0}$ for some $\alpha_0 \in A$ such that $\rho^{\alpha_0}_{3}(x)=\cdots =\rho^{\alpha_0}_{m^2}(x)> 0$. 
\end{enumerate}

In case \ref{casea}, by the results in \cite[Theorem 2.1] {NiTam} and \cite[Corollary 2.3]{NiWu}, if we run the K\"ahler--Ricci flow with the initial metric $g^T_{\alpha_0}$ in $V_{\alpha_0}$, then the K\"ahler curvature operator $R^T_{\alpha_0}(t)$ of the evolved metric $g^T_{\alpha_0}(t)$  is 2-positive in $V_{\alpha_0}$ for $t>0$. Now, for the open cover $\{U_\alpha\}_{\alpha\in A}$ of $M$, we can find a finite subcover $\{U_\alpha\}_{\alpha\in A'}$ since $M$ is compact, and we add $\alpha_0$ into $A'$. Assume $U_\alpha \cap U_{\alpha_0}\neq \emptyset$ for some $\alpha \in A'$. We run the K\"ahler--Ricci flow with initial metric $g^T_\alpha$ in $V_\alpha$. Since the metrics $\{g^T_\alpha\}$ are isometric on the overlaps, the K\"ahler curvature operator of the evolved metric $g^T_\alpha(t)$ is 2-positive in $\pi_\alpha(U_\alpha \cap U_{\alpha_0})$. So, for slightly larger $t$, the K\"ahler curvature operator of $g^T_\alpha(t)$ is 2-positive in $V_\alpha$ by \cite[Theorem 2.1] {NiTam} and \cite[Corollary 2.3]{NiWu}. Since $A'$ is a finite set, we eventually have a transverse K\"ahler metric whose K\"ahler curvature operator is 2-positive, so the transverse orthogonal bisetional curvature is positive by \cite[Lemma 2.10]{YangZhang}. By Theorem \ref{thm He--sun positive},  the universal cover of $(M,g)$ is isomorphic to a weighted Sasaki sphere with a simple metric.

In case \ref{caseb}, the universal cover of $(M,g)$ is isomorphic to $\mathbb{R}^{2m+1}$ with the standard Sasaki structure according to the classification of  Tanno in \cite{Tanno} (see also \cite[Sections 7.3 and 7.4]{Blair}).  

In case \ref{casec}, we claim that there are two situations:
   \begin{enumerate}
       \item The universal cover of $(M,g)$ is isomorphic to a weighted Sasaki sphere with a simple metric.
       
\item The universal cover $\widetilde{M/F_{\xi}}$ of the quotient orbifold $M/F_{\xi}$ is isomorphic to $W\mathbb{P}^1\times W\mathbb{P}^1$ or $W\mathbb{P}^m$.
   \end{enumerate} 
Indeed, the K\"ahler curvature operator is nonnegative, so the transverse bisectional curvature is nonnegative. By Theorem \ref{thm He--Sun nonnegative}, there are two cases. If $M$ is irregular, then the universial cover of $(M,g)$ is isomorphic to a weighted Sasaki sphere with a simple metric. If $M$ is quasi-regular, then the universal cover $\widetilde{M/F_{\xi}}$ of the quotient orbifold $M/F_{\xi}$ is isomorphic to the product of the following: 
\begin{enumerate}
    \item the weighted projective space $W\mathbb{P}^k$ with nonnegative bisectional curvature.
    
    \item compact irreducible Hermitian symmetric space of rank $\geq 2$ with the canonical metric.
    
    \item $\mathbb{C}^k$ with the flat metric.
    \end{enumerate}

Suppose $\widetilde{M/F_{\xi}}$ is isomorphic to a product $M_1\times M_2$ where $\dim_\mathbb{C}M_1=k $. We may assume $\dim_\mathbb{C}M_1\geq \dim_\mathbb{C}M_2$, so $k\geq m-k$. The number of nonzero eigenvalues of the K\"ahler curvature operator of $M_1\times M_2$ is  at most $k^2+(m-k)^2$. Meanwhile, $\widetilde{M/F_{\xi}}$ has $m^2-2$ nonzero eigenvalues at some point, so  $m^2-2\leq k^2+(m-k)^2$. We get $k(m-k)\leq 1$, hence $(m,k)=(2,1)$ and $\dim_\mathbb{C}M_1=\dim_\mathbb{C}M_2=1$. The only compact irreducible Hermitian symmetric space of complex dimension one is $\mathbb{CP}^1$ but it has rank one, so we can rule out this case. Neither $M_1$ or $M_2$ can be $\mathbb{C}$ with the flat metric because $M_1\times M_2$ has $2^2-2=2$ nonzero eigenvalues at some point. Therefore, the only possibility for $\widetilde{M/F_{\xi}}$ is $W\mathbb{P}^1\times W\mathbb{P}^1$.

Suppose $\widetilde{M/F_{\xi}}$ is indecomposable. The case $\mathbb{C}^m$ with the flat metric cannot happen since $\widetilde{M/F_{\xi}}$ has $m^2-2$ nonzero eigenvalues at some point. If $\widetilde{M/F_{\xi}}$ is a compact irreducible Hermitian symmetric space, then $\widetilde{M/F_{\xi}}$ is biholomorphic to $\mathbb{CP}^m$ (see \cite[Page 689, line -2]{YangZhang}). \end{proof}


\bibliographystyle{amsalpha}
\bibliography{dom}

\end{document}